\documentclass[oneside,english]{amsart}
\usepackage[latin1]{inputenc}
\usepackage{amssymb}

\usepackage{color}
\usepackage{epsfig}

\makeatletter

\theoremstyle{plain}
 \newtheorem{thm}{Theorem}[section]
 \numberwithin{equation}{section} 
 \numberwithin{figure}{section} 
 \theoremstyle{remark}
 \newtheorem*{acknowledgement*}{Acknowledgement}
 \theoremstyle{plain}
 \theoremstyle{plain}
 \theoremstyle{plain}
 \newtheorem{lem}[thm]{Lemma} 
 \theoremstyle{definition}
  
 \theoremstyle{plain}
 \newtheorem{prop}[thm]{Proposition} 
 \theoremstyle{plain}
 \newtheorem*{thm*}{Theorem}

\usepackage{color}
\usepackage{epsfig}
\newcommand{\hide}[1]{}
\newcommand{\ZZ}{\mathbb{Z}}
\newcommand{\NN}{\mathbb{N}}
\newcommand{\ZD}{{\mathbb{Z}^d}}

\newcommand{\htop}{h_{\mathit{top}}}
\newcommand{\TSD}{time-space diagram}

\usepackage{babel}
\makeatother

\begin{document}

\title{Growth-type invariants for $\ZD$ subshifts of finite type and classes arithmetical of real numbers}
\author{Tom Meyerovitch}
\begin{abstract}
We discuss some numerical invariants of multidimensional shifts of
finite type (SFTs) which are associated with the growth rates of the
number of admissible finite configurations. Extending an unpublished
example of Tsirelson \cite{tsirelson_92}, we show that growth
complexities of the form $\exp(n^\alpha)$ are possible for
non-integer $\alpha$'s. In terminology of \cite{carvalho_97}, such
subshifts have entropy dimension $\alpha$. The class of possible
$\alpha$'s are identified in terms of arithmetical classes of real
numbers of Weihrauch and Zheng \cite{ZW01}.
\end{abstract}

\maketitle
\section{Introduction}
A \emph{multidimensional shift space}, or a \emph{$\ZD$-subshift},
is a subset of $S^\ZD$, defined by a (possibly infinite) list of
translations invariant local rules. When there is a finite list of
rules defining $X$, it is a \emph{shift of finite type (SFT)}. We
always assume $S$ is a finite set, and refer to it as the
\emph{alphabet} or \emph{set of tiles} of the subshift.

Symbolic dynamics is concerned with the study of subshifts, and in
particular SFTs. Apart from the pure mathematical interest in these
objects, they come up as natural models in various fields such as
statistical mechanics and computer science.

There is a very sharp contrast in the behavior of SFTs in dimension
$d=1$ with respect to higher dimensions $d>1$. Quoting Klaus Schmidt
\cite{schmidt_76}: ``$\ldots$ Higher dimensional Markov shifts (is)
a difficult field of research with no indication yet of satisfactory
general results. This lack of progress is all the more remarkable
when compared with the richness of the theory in one dimension.'' In
this paper we attempt to bridge a certain aspect of our
understanding of multidimensional SFTs, in view of recent progress
in this field. Namely, we study asymptotic growth properties of
SFTs. A step in this direction was taken in \cite{HM_SFT}, where
there most classical growth type -- exponential -- was investigated.
We discover that when studying various growth types, different
arithmetical classes of real number come up.  As SFTs are seemingly
reasonable models for ``physical systems'', their growth-type
invariant can be considered ``physical constants''. Our present
result of  is a characterization of the possible
``broken-exponential'' asymptotic growth rates of SFTs- namely those
$\alpha$'s for which the number of admissible $n$-blocks grows
exponentially in $n^\alpha$. In particular, it turns out that there
are cases where these constants can not be obtained as a limit of
any computable sequence.

\textbf{Acknowledgments:} This work is a part of the author's Ph.D
thesis, written in Tel-Aviv University under the supervision of
Professor Aaronson. The author acknowledges the support of support
of the Crown Family Foundation Doctoral Fellowships ,USA. The author
thanks Professor Tsirelson for making his unpublished result of
\cite{tsirelson_92} available to him.

\section{Growth-type invariants for subshifts}

Let $X$ be a $\ZD$ subshift. For $k \ge 1$, denote by $N_k(X)$ the
number of $X$-admissible $k \times k$ configurations. The sequence
$\{N_k(X)\}_{k \ge 1}$ by itself is not an invariant of SFT
isomorphism.

A function of the form
$$I(X) = \limsup_{k \to \infty}f_k(N_k(X)),$$
or
$$I(X) = \liminf_{k \to \infty}f_k(N_k(X)),$$

where $\{f_k\}$ is a sequence of real valued functions of the
positive-integers, is called a \emph{growth-type invariant}, if
$I(X) = I(Y)$ for any topologically conjugate $\ZD$ subshifts $X$
and $Y$.

The fundamental example of a growth-type invariant is the
\emph{topological entropy}, defined by
$$\htop(X)=\lim_{k \to \infty}\frac{\log N_k(X)}{k^d}.$$

It is natural to investigate the possible asymptotic behaviors of
$N_k(X)$  where $X$ is a subshift of finite type. A result of this
type was obtained in \cite{HM_SFT}: The possible values for the
topological entropy of $\ZD$-SFTs is precisely the class of
non-negative right recursively enumerable numbers.

For a perspective, lets discuss the much simpler situation in $d=1$. 
Any one-dimensional SFT $X$ can be represented by a non-negative
integer valued matrix $A$.
The numbers $N_k(X)$ are expressed by the sum of the entries of
$A^k$. If the spectral radius of $A$ is strictly greater then $1$,
$N_k(X)$ grows exponentially, at a rate determined by the spectral
radius of $A$, which can  be any Perron number, as shown by Lind
\cite{lind84}. Otherwise, the spectral radius of $A$ is equal to
$1$, and so $N_k(X)$ is a polynomial in $k$. For any integer $N$, it
is easy to construct a $1$-dimensional subshift with $N_k(X)$ a
polynomial or degree $N$.  ``Intermediate growth'' and ``non-integer
polynomial degree'' are impossible for SFTs in dimension $1$.

 Here are a some growth-type invariants which will be
studied in this paper:

\begin{itemize}
\item{
The upper and lower \emph{entropy dimensions} of a  subshift $X
\subset \Sigma^{\ZD}$ are:
\begin{equation}\label{eq:upper_entropy_dimension} \overline{D}(X)=
\limsup_{k \to \infty} \frac{\log \left(\log N_k(X) \right)}{\log
k}\end{equation} and
\begin{equation}\label{eq:lower_entropy_dimension}\underline{D}(X)= \liminf_{k \to
\infty} \frac{\log \left(\log N_k(X) \right)}{\log k}.\end{equation}

When there is equality of the upper and lower entropy dimension, we
say that $X$ has entropy dimension $D(X)$, in which case:
$$D(X)= \lim_{k \to \infty} \frac{\log \left(\log N_k(X) \right)}{\log
k}.$$ }
\item{
The (upper/ lower) \emph{polynomial growth type} of an SFT $X$, are:
$$\overline{P}(X)= \limsup_{k \to \infty} \frac{\log N_k(X)}{\log
k},~ \underline{P}(X)= \liminf_{k \to \infty} \frac{\log
N_k(X)}{\log k},$$

and $$P(X)=\lim_{k \to \infty} \frac{\log N_k(X)}{\log k},$$  when
there is a limit.}
\end{itemize}

The above quantities are indeed invariants:
\begin{lem}\label{lem:entropy_dim_inv}
The upper and lower entropy dimensions  and polynomial growth types
are all invariants of topological congruency.
\end{lem}
\begin{proof}
  Suppose $\phi:X \to  Y$ is a congruency of $X$ and $Y$. We can assume that both $\phi$ and $\phi^{-1}$ are given by  $K$-block maps.
 Under this assumption, for every $k \ge K$,
 $$N_{k-K}(X) \le N_{k}(Y) \le N_{k+K}(X).$$

 Since $\lim_{k \to \infty}\frac {\log( k+K)}{\log k}=1$, it follows that
 $$\overline{P}(X)=
  \limsup_{k \to \infty}\frac{ \log N_{k -K}(X)}{\log k} \le \overline{P}(Y).$$
 Replacing the roles of $X$ and $Y$, it follows that $\overline{P}(Y) \le \overline{P}(X)$.

 The same sort of argument proves that the lower polynomial growth rates,
  the upper entropy dimension and the lower entropy dimension are all invariants of topological congruency.
\end{proof}

\section{\label{sec:examples}Constructions and illustrative examples of subshifts}
In this section we describe families of subshifts having various
growth-type properties. Our main construction scheme, described in
the next section, is a construction which combines ingredients from
the SFTs described in this section.


\subsection{\label{subsec:p_adic_SFTs}Robinson-type $p$-adic SFTs}
We describe a family of $\ZZ^2$-SFTs, which are topologically almost
$\ZZ^2$ $p$-adic odometers. This construction 
is a variant of Robinson's. 

 We
denote the SFT described in \cite{robinson71} by $X_2$. For this
subshift, the alphabet consists of all possible reflections and
rotations of the $5$ tiles shown in figure \ref{fig:robinson_SFT},
along with extra parity markings in $(\ZZ / 2\ZZ) ^2$.  The left
most tile in figure \ref{fig:robinson_SFT} is called a \emph{cross},
and the other $4$ tiles are called \emph{arms}. There are $4$
possible orientations for a cross, which called $SW$, $SE$, $NE$ and
$NW$. The basic restriction on these tiles is that any arrow head
must meet any arrow tail. The restrictions on the parity tiles are
that if $x_{i,j}$ has parity $(n,m)$ then $x_{i+1,j}$ and
$x_{i,j+1}$ have parity markings $(n+1,m)$ and $(n,m+1)$
respectively, and that parity $(0,0)$ can appear only on a cross.

\begin{figure}\label{fig:robinson_SFT}
\includegraphics[width=0.8\textwidth,clip]{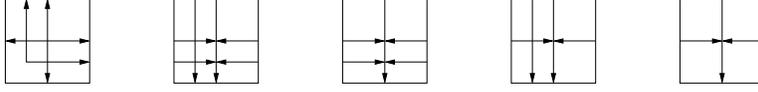}
\caption{The basic tiles of Robinson's SFT}
\end{figure}

It was proved in \cite{robinson71} that this subshift admits no
periodic points. Furthermore it is almost a $\ZZ^2$ $2$-adic
Odometer- in a sense explained in the following.

We now describe a base $p$-analog of this construction for a natural
number $p
>2$.  Let $(p):=\ZZ / p\ZZ$ denote the
cyclic group of order $p$ and $\ZZ_p$ denote the $p$-adic integers.
We define an SFT $X_p \subset S_p^{\ZZ^2}$ as follows: The alphabet
$S_p$ will consist of \emph{nodes} which are elements of $\{B,W\}
\times \left((p) \setminus \{1\}\right)^2$, and of \emph{arrows}
which are elements in $\{\uparrow,\rightarrow\} \times (p)^2$.

The adjacency rules of $X_p$ are the following: For any  $x \in
X_p$,
\begin{enumerate}
  \item{ If $x_{i,j} = (B,n,m)$, and $n \ne 0$,  then $x_{i + 1, j}= (B, n
+ 1,m)$.}
\item{ If $x_{i,j}=(B,0,m)$ then
$x_{i+1,j}=(\uparrow, n',m')$ for some $n',m' \in (p)$.}
\item{If $x_{i,j} = (B,n,m)$, and $m \ne 0$,  then $x_{i , j +1}= (B, n
,m+1)$.}

\item{If $x_{i,j}=(B,n,0)$ then
$x_{i,j+1}=(\rightarrow, n',m')$ for some $n',m' \in (p)$. }
\item{ If $x_{i,j}=(B,0,m)$ then
$x_{i+2,j}=(B,2,m)$.}
\item{If $x_{i,j}=(B,n,0)$ then $x_{i,j+2}=(B,n,2)$.}
\item{If $x_{i,j}=(W,n,m)$ then $x_{i+1,j}=(\rightarrow,n,m)$}
\item{If $x_{i,j}=(W,n,m)$ then $x_{i,j+1}=(\uparrow,n,m)$}
\item{If $x_{i,j}=(\uparrow,n,m)$ and $m \ne 0 \mod p$ then either $x_{i,j+1}=(\uparrow,n,m)$
or $x_{i,j+1}=(W,n,m+1)$.}
\item{If $x_{i,j}=(\uparrow,n,0)$ then either $x_{i,j+1}=(\uparrow,n,0)$
or $x_{i,j+1}=(\rightarrow,n',0)$ for some $n' \in (p)$ and
$x_{i,j+2}=(\uparrow,n,1)$.}
\item{If $x_{i,j}=(\rightarrow,n,m)$ and $n \ne 0 \mod p$ then either $x_{i+1,j}=(\rightarrow,n,m)$
or $x_{i,j+1}=(W,n+1,m)$.}
\item{If $x_{i,j}=(\rightarrow,0,m)$ then either $x_{i+1,j}=(\rightarrow,0,m)$
or $x_{i+1,j}=(\uparrow,n',m')$ for some $n',m' \in (p)$ and
$x_{i+2,j}=(\rightarrow,1,m)$.}
\item{If $x_{i,j}=(\uparrow,n,m)$ then $x_{i+1,j}=(B,n',l)$ or $x_{i+1,j}=(\rightarrow,n,m')$ for $n',m',l
\in (p)$.}
\item{If $x_{i,j}=(\rightarrow,n,m)$ then $x_{i,j+1}=(B,l,m')$ or $x_{i,j+1}=(\uparrow,m')$ for $l,m'
\in (p)$.}
\end{enumerate}

Consider a $p\times p$ admissible configuration in $X_p$ whose
bottom-left corner is $(B,0,0)$. Such a configuration is called a
\emph{$1$-cluster}. The labeling of a $1$-cluster is
defined to be the labeling of position $(1,1)$. For example with
$p=2$, omitting the labels on the arrows and the $B/W$ markings, a
level-$1$ cluster with labeling $t_1$ looks like this:
\begin{center}
\begin{tabular}{||c|c|c||}
  \hline\hline
 (0,2) & $\uparrow$ & (2,2)\\\hline
$\rightarrow$ & $t_1$ & $\rightarrow$ \\\hline
(0,0) & $\uparrow$ & (2,0) \\
  \hline\hline
\end{tabular}
\end{center}

Inductively, an \emph{$n$-cluster} is a $p^n\times p^n$ configuration
in $X_p$, whose lower-left $p^{n-1}\times p^{n-1}$ sub-configuration
is an $(n-1)$-cluster with labeling $(W,0,0)$ or $(B,0,0)$. For
example, here is a $2$-cluster, again with $p=3$:

\begin{center}
\begin{tabular}{||c|c|c||c|c|c||c|c|c||}
\hline \hline (0,2) & $\uparrow$ & (2,2) & (0,2) & $\uparrow$ &
(2,2) & (0,2) & $\uparrow$ & (2,2)\\
 \hline $\rightarrow$ & (0,2) &
$\rightarrow$ & $\rightarrow$ & $\uparrow$ & $\rightarrow$ &
$\rightarrow$ & (2,2) & $\rightarrow$\\ \hline

(0,0) & $\uparrow$ & $(2,0)$ &  (0,0) & $\uparrow$ & (2,0) &  (0,0) & $\uparrow$ & (2,0)\\
\hline \hline (0,2) & $\uparrow$ & (2,2) & (0,2) & $\uparrow$ &
(2,2) & (0,2) & $\uparrow$ & (2,2)\\
 \hline
 $\rightarrow$ & $\rightarrow$ &
$\rightarrow$ & $\rightarrow$ & $t_2$ & $\rightarrow$ &
$\rightarrow$ & $\rightarrow$ & $\rightarrow$\\ \hline (0,0) &
$\uparrow$ &
$(2,0)$ & (0,0) & $\uparrow$ & (2,0) &  (0,0) & $\uparrow$ & (2,0)\\
\hline \hline (0,2) & $\uparrow$ & (2,2) & (0,2) & $\uparrow$ &
(2,2) & (0,2) & $\uparrow$ & (2,2)\\ \hline $\rightarrow$ & (0,0) &
$\rightarrow$ & $\rightarrow$ & $\uparrow$ & $\rightarrow$ &
$\rightarrow$ & (2,0) & $\rightarrow$\\ \hline
(0,0) & $\uparrow$ & $(2,0)$ &  (0,0) & $\uparrow$ & (2,0) &  (0,0) & $\uparrow$ & (2,0)\\
\hline\hline
\end{tabular}
\end{center}
$n$-clusters appear in any $x \in X_p$, and determine a
$p$-adic structure. A more precise statement is the following:

\begin{prop}
  As a topological $\ZZ^2$ system, the SFT $X_p$ described above is
  ``almost'' conjugate to the action of $\ZZ^2$ on $\ZZ_p^2$ by
  $T_{n,m}(x,y)=(x+n,y+m)$, with respect to the standard compact topology on $\ZZ^2_p$.
   Namely, there exist a factor map 
  $\pi:X_p \to \ZZ_p^2$, 
  which is injective on a dense orbit
  in $X_p$, and is at most $2p^2$-to-one.
\end{prop}
\begin{proof}
For $p=2$ a proof is given in section $8$ of Robinson's paper
\cite{robinson71}. For $p > 2$, we 
describe an almost $1-1$ factor map $\pi: X_p \mapsto \ZZ_p^2$: Let
$x \in X_p$. By a straightforward case study of the rules above, it
follows that there exists unique indices $(i_0,j_0) \in (p)^2$ such
that $x_{n,m}=(B,n+i_0,m+i_0)$ whenever $n+i_0 \ne 1 \mod p$ and
$m+i_0 \ne 1 \mod p$. Proceeding by induction, one verifies that for
any $k \ge 1$, there are $(i_k,j_k) \in (p)^2$ such that
$$x_{p^kn-r_k+\sum_{j=0}^{k-1}p^j,p^km-s_k+\sum_{j=0}^{k-1}p^j}=(W,n+i_k,m+j_k),$$
 whenever $n+i_k \ne 1 \mod
p$ and $m+j_k \ne 1 \mod p$, where $r_k= \sum_{t=0}^{k-1}i_t$ and
$s_k= \sum_{t=0}^{k-1}j_t$.

Let $r_\infty = \sum_{k=0}^{\infty}p^ki_k
\in \ZZ_p$, $s_\infty = \sum_{k=0}^{\infty}p^kj_k \in \ZZ_p$.

Define: 
$$\pi(x):=(r_\infty,s_\infty)$$

For any $k$ the pair $(r_k,s_k)$ can obtain any value in
$(p^k)^2$. This can be seen by examining all possible
$k$-clusters in a given $(k+1)$-cluster. Thus, the map $\pi$ is
onto $\ZZ_p^2$. Also ,
shifting $x$ by $(m,n) \in \ZZ$ results in adding $(m,n)$ to $\pi(x)$, where addition is in $\ZZ_p^2$.
 Since $\pi$ is clearly continuous, it is indeed a topological factor map of $X$ onto $\ZZ_p^2$.

It remains to show that $\pi$ is almost-one to one , and at most
$2p^2$-to-1: Suppose $\pi(x)=(r,s)$. Let $r'=r- \sum_{j=0}^\infty
p^j$, $s'=s- \sum_{j=0}^\infty p^j$. Let $m,n \in \ZZ$. Suppose
$m-r' \ne 0$ let $k_r$ be the smallest integer such that $m-r' \ne 0
\mod p^k_s$. Similarly, if $n-s' \ne 0$ let $k_s$ be the smallest
integer such that $n-s' \ne 0 \mod p^k_r$. From the definition of
$\pi$ it follows that
$$x_{m,n}=\begin{cases}
(B,m-r',n-s') & \mbox{if } k_s=k_r=0\\
(W,(m-r')/p^{k-1},(n-s')/p^{k-1}) & \mbox{if } k_s=k_r=k >0\\
(\uparrow,(m-r')/p^{k-1},(n-s')/p^{k-1}) & \mbox{if } k_r>k_s\\
(\rightarrow,(m-r')/p^{k-1},(n-s')/p^{k-1}) & \mbox{if } k_s>k_r\\
\end{cases}$$

Thus, $x_{m,n}$ is determined unless $n = s'$ or $m = r'$. If $n =
s'$ and $m \ne r'$ then there exists $i_{\infty} \in (p)$ with
$x_{m,n}=(\rightarrow,i_\infty)$ or
$x_{m,n}=(\rightarrow,i_\infty+1)$. If $n \ne s'$ and $m = r'$ then
there exists $j_{\infty} \in (p)$ with $x_{m,n}=(\uparrow,j_\infty)$
or $x_{m,n}=(\uparrow,j_\infty+1)$. $x_{i_\infty,j_\infty}$ is
either $(\uparrow,j_\infty)$ or $(\rightarrow,i_\infty)$. We
conclude that $\pi$ is at most $2p^2$ to 1, and any point $(s,t) \in
\ZZ_p^2$ such that $s'$ and $r'$ are not ordinary integers has a
unique preimage.

\end{proof}

For this subshift, it not difficult to calculate that
$N_{k}(X_p)=k^2\left((p-1)^2+2p\right)$. In particular,
$\htop(X_p)=D(X_p)=0$ and $P(X)=2$.
\subsection{\label{subsec:frac_entrop_dim}Subshifts with fractional entropy dimension}
We will describe a recipe for a $\ZZ^2$-subshift (not of finite
type!) with entropy-dimension $\alpha\in (0,2)$, for an arbitrary
real number $\alpha$ in the range . The main construction, is an
``SFT implementation'' of this type of subshift, for a certain
family of $\alpha$'s.

Given any number $\alpha \in (0,2)$, Choose a sequence $(a_n)_{n
\in\NN} \in \{0,1\}^\NN$ such that $\alpha = \lim_{n \to
\infty}\frac{2}{n}\sum_{k=0}^n a_k$. In other words, $(a_n)_{n
\in\NN}$ is the indicator of a set of integers with asymptotic
density $\alpha$.

Now let $Y_\alpha$ be the following  (non-SFT!) extension of
Robinson's SFT $X_2$: Allow each cross can be blue or white, subject
to the following restrictions:
\begin{itemize}
  \item{Any cross which is inside a cluster with white central cross must be white.}
  \item{The central cross of an $n$-cluster can be blue only if it is labeled $SW$ or if $a_n=1$.}
\end{itemize}

Now allow each blue cross to contain another marking, which can have one of two values. We refer to these values as
 ``head''  and ``tail''. No further restrictions on the head/tail marking of blue crosses are imposed.

We claim that $Y_\alpha$ has entropy dimension $\alpha$. To see this, let $s_n$ denote the number of possibilities for
blue/white and head/tail markings of a $Y_\alpha$-admissible $n$-cluster.
Consider a  $Y_\alpha$-admissible $n$-cluster.

If the central cross is white, then  all the markings are determined.
Otherwise,  the central cross can have head or tail marking and
according to the value of $a_n$, there are either $4$ or $1$
$(n-1)$-clusters to be determined. Thus,

$$s_{n+1} = \begin{cases}
  s_n +1 & \mbox{if }a_n=0\\
  2s_n^4+1 & \mbox{if }a_n=1\\
\end{cases}$$

since $\alpha >0$, it follows that $a_n=1$ for infinitely many $n$'s, and so $s_n \to \infty$.
It follows that for any $\epsilon>0$ and sufficiently large $n$'s:

$$ \begin{cases}
  \log s_n \le \log s_{n+1} \le (1+\epsilon)\log s_n  & \mbox{if }a_n=0\\
  4\log s_n \le \log s_{n+1} \le (4+\epsilon)\log s_{n} & \mbox{if }a_n=1\\
\end{cases}$$

Thus, by taking $\log$'s on both sides, it follows by induction that
for sufficiently large $N$'s, and $n >N$,
$$\log4 \sum_{k=N}^n a_k \le \log \log s_n \le \log(4+\epsilon) \sum_{k=N}^n
a_k,$$ and so:
$$\log \log s_n \sim \log 4 \sum_{k=1}^n a_k.$$

As any admissible $k \times k$ configuration in $Y_\alpha$ contains a $(\lfloor \log_2 k \rfloor -1)$-cluster and is
contained in a $(\lceil \log_2 k \rceil +1)$-cluster, it follows that the entropy dimension of $Y_\alpha$ is
$$D(Y_\alpha) = \lim_{n \to \infty}\frac{\log \log s_n}{n\log 2} = 2 \lim_{n \to \infty}\frac{1}{n}\sum_{k=0}^n a_n.$$

More generally, given $0 \le \alpha_1 \le \alpha_2 \le 2$, we could
have chosen $(a_n)_{n \in\NN}$ to be an indicator of a set with
lower density $\alpha_1$ and upper density $\alpha_2$. In this case,
we could construct a subshift $Y_{\alpha_1,\alpha_2}$ with lower
entropy-dimension $\alpha_1$ and upper entropy-dimension $\alpha_2$.

\subsection{\label{subsec:rat_entropy_dim} SFT's with entropy dimension $\frac{\log q}{\log
p}$:} The following class of examples is a variation on an
unpublished construction of B. Tsirelson \cite{tsirelson_92} for
$\ZZ^2$-SFTs with fractional entropy dimensions.

Lets first present an illustration of the idea : Figure
\ref{fig:tsirelson_SFT} is a $2^4 \times 2^4$ admissible
configuration for Tsirelson's SFT. Ignoring the cross on the upper
right corner, there are $3^0+3^1+3^2$ blue crosses. Similarly, such
a $(2^n-1) \times (2^n-1)$ configuration could contain at most
$\sum_{k=0}^{n-2}3^k$ blue crosses. Allowing each blue cross an
extra mark in $\{\mathit{head},\mathit{tail}\}$ independently, the
entropy dimension is $\frac{\log 3}{\log 4}$.

\begin{figure}\label{fig:tsirelson_SFT}
\includegraphics[width=0.8\textwidth,clip]{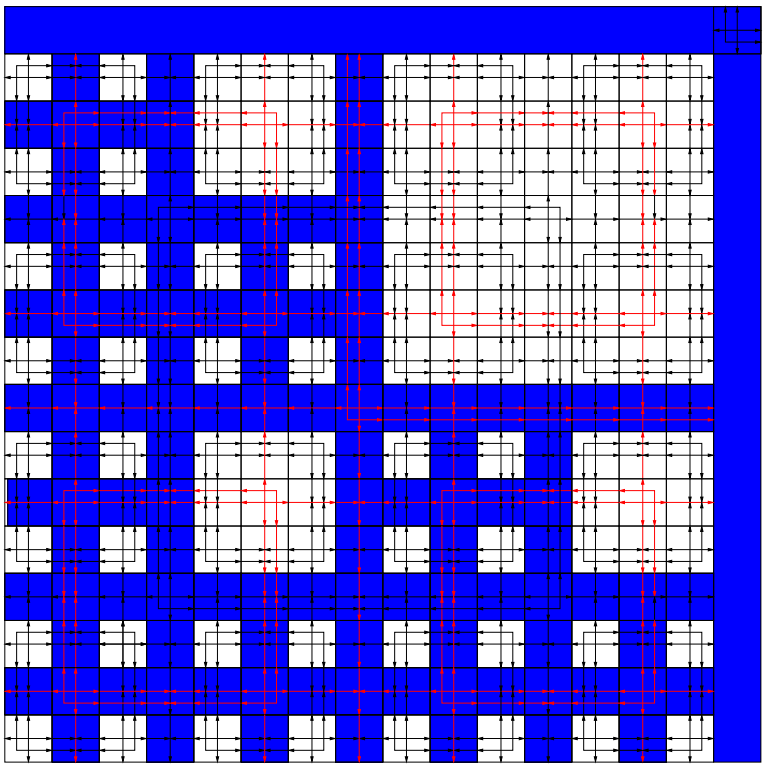}
\caption{Tsirelson's SFT with entropy dimension $\frac{\log 3}{\log 2}$}
\end{figure}

Now for a formal description of a more general scheme: Let $2 < p$
and $0< q < 2p$ be integers. We introduce an extension $Y=Y_{p,q}$
of the SFT $X_p$, which has fractional entropy dimension. The tiles
will be marked with an extra color which can be blue or white,
subject to the restriction specified below.
%

Choose a subsets $$B_h,B_v \subset \{0,2,\ldots,p-1\},$$ with $|B_h|+|B_v|=q$. Impose
the restriction that a node $(W,x,y)$ can be marked blue iff $x \in B_h$ and $y
\in B_v$. Further impose the restriction that arrows originating at a
blue node must be marked blue, and that whenever an arrow head meets
an arrow tail either both must be blue or both not blue. Finally,
impose the restriction that whenever a blue arrowhead meets a
perpendicular arrow, this arrow must also be blue.

It follows that a node can be blue only if the central node of any
cluster containing this node is indexed by $(x,y) \in A^2$. Thus,
the number of blue nodes in an $n$-cluster is at most
$b_n=\sum_{k=1}^n q^k$, so $\log b_n \sim n\log q$. Now we allow an
extra marking in $\{0,1\}$ for each blue nodes, without any further
restrictions. There will be in total $s_n= p^2\sum_{j=1}^{b_n}2^j$
$n$-clusters, and so $\log \log s_n \sim \log b_n \sim  n \log q$.
Since any admissible $k\ \times k$ configuration in $Y$ is contained
in a $(\log_p k +1)$-cluster, it follows that $$\log \log N_k(Y)
\sim \log \log s_{\log_p k} \sim  \frac{\log q}{\log (p-1)} \log
k.$$ Thus, $Y$ has entropy dimension $\log q / \log (p-1)$.

%
\subsection{\label{subsec:example_polynomial_growth}SFT with polynomial growth type $2+\frac{\log M}{\log p}$}

A different kind of extension for the SFT $X_p$ defined in
\ref{subsec:p_adic_SFTs}  will yield an SFT $X_{p,M} \subset
\left(S_p \times M\right)^{\ZZ^2}$ which has polynomial growth-type
$2+\frac{\log M}{\log p}$ for any non-negative integer $M$. The
local restrictions will guarantee that the extra marking in $(M)$ is
constant within each level, and independent between levels: A white
node has the same $(M)$-color as the arrows pointing towards it and
away from it, and any arrow has the same $(M)$-marking as its
predecessor (at distance $1$ or $2$). An $n$-cluster will have $M^n$
possible $(M)$-markings, so $\log N_{p^k}(X_{p,T}) \sim k(\log M + 2
\log p) $.

It follows that $P(X_{p,m})=2+\frac{\log M}{\log p}$.

\section{The Zheng-Weihrauch hierarchy of real numbers}
In order to state our main result, we briefly recall from
\cite{ZW01} Weihrauch and Zheng's hierarchy of real numbers. Denote
by $\Gamma_{\mathbb{Q}}$ the set of computable functions
$f:\mathbb{N}^k \to \mathbb{Q}$ (there exist a Turing machine
$\mathcal{R}_f$ which terminates with output $f(n_1,\ldots,n_k)$,
given the input $(n_1,\ldots,n_k) \in \mathbb{N}^k$). For $n \ge 1$,
the classes of real numbers $\Sigma_n,\Pi_n$ and $\Delta_n$ are
defined as follows:
\begin{enumerate}
\item{
$$\Sigma_n :=\{x \in \mathbb{R}:~
(\exists f \in \Gamma_\mathbb{Q})x=\sup_{i_1}\inf_{i_2}\ldots
\Theta_{i_n}f(i_1,\ldots,i_n)\}$$
$$\Pi_n :=\{x \in \mathbb{R}:~
(\exists f \in \Gamma_\mathbb{Q})x=\inf_{i_1}\sup_{i_2}\ldots
\overline{\Theta}_{i_n}f(i_1,\ldots,i_n)\}$$}
\item{$\Delta_n := \Sigma_n \cap \Pi_n$}.
\end{enumerate}
In the above $\Theta$ and $\overline{\Theta}$ stand for $\sup$ or
$\inf$ according to the parity of $n$.

It is known that for any $n \ge 0$, the inclusions $\Pi_n \subset
\Sigma_{n+1}$ and $\Sigma_n \subset \Pi_{n+1}$ are proper.

In particular, we will be interested in the first few classes in
this hierarchy, and some of their equivalent characterizations
proved in \cite{ZW01}:
\begin{itemize}
  \item{$x\in \Sigma_1$ ($\Pi_1$) iff $x=\sup_{n \in \NN}f(n)$ ($x=\inf_{n \in \NN}f(n)$)}
  \item{$x \in \Sigma_2$ ($\Pi_2$) iff $x=\sup_{n \in \NN}\inf_{m
  \in \NN}g(n,m)$ ($x=\inf_{n \in \NN}\sup_{m
  \in \NN}g(n,m)$) iff $x=\liminf_{n \to \infty}f(n)$ ($x=\limsup_{n \to
  \infty}f(n)$).}
  \item{$x \in \Delta_1$ iff $x=\lim_{n \to \infty}f(n)$
  \emph{effectively}. This means $|f(n)-x| \le 2^{-n}$.}
  \item{$x \in \Delta_2$ iff $x=\lim_{n \to \infty}f(n)$.}
\end{itemize}

 where $f:\NN \to
\mathbb{Q}$ and $g:\NN^2 \to \mathbb{Q}$ are computable functions.

This arithmetical hierarchy of real numbers recently already made an
appearance in symbolic dynamics. The main result of \cite{HM_SFT} is
that for $d \ge 2$, the class of entropies of $\ZD$ SFT's is
$[0,+\infty) \cap \Pi_1$. This was followed by a subsequent result
of Hochman \cite{hochman1}: For $d \ge 3$, the class of entropies of
$\ZD$ cellular automata is $[0,+\infty] \cap \Delta_2$.

For our purposes, it will be convenient to associate a binary
function with a real number in $\Pi_3$ (and $\Delta_2$, $\Sigma_2$):
\begin{lem}\label{lem:binary_avg_comp}
\begin{enumerate}
  \item{ $x \in [0,1]$ is in $\Pi_3$ iff $x$ is the
upper-density for some set of integers  whose complement is
recursively enumerable. In other words,
  there exists a computable function $g:\mathbb{N} \times  \mathbb{N} \to  \{0,1\}$ such that
  $$x=\limsup_{n \to \infty}
  \frac{1}{n}\sum_{k=1}^n\inf_{j} g(k,j).$$}
  \item{$y \in [0,1]$ is in $\Sigma_2$ iff it is the
  lower-density of some set of integers  whose complement is
recursively enumerable. In other words,
$$y=\liminf_{n \to \infty}
  \frac{1}{n}\sum_{k=1}^n \inf_{j} f(k,j).$$
  For some computable  $f:\mathbb{N} \times  \mathbb{N} \to  \{0,1\}$.}
   \item{$z \in [0,1]$ is in $\Delta_2$ iff it is the
  density of some set of integers whose complement is
recursively enumerable. In other words,
$$z=\lim_{n \to \infty}
  \frac{1}{n}\sum_{k=1}^n \inf_{j} f(k,j).$$
  For some computable  $f:\mathbb{N} \times  \mathbb{N} \to  \{0,1\}$.}
\end{enumerate}
\end{lem}

\begin{proof}We prove the part about $x \in \Pi_3 \cap [0,1]$, as the proofs of the other parts are completely analogous.
By \cite{ZW01}, for $x$ as above, we have
$$x=\limsup_{n \to \infty}\inf_j G(n,j),$$
 for some recursive function $G:\mathbb{N} \times \mathbb{N} \to \mathbb{Q}\cap [0,d]$.
By replacing $G(n,j)$ with $\min_{k\le j}G(n,j)$, is a , we can
assume that $G$ is monotone in the second variable.

Let $t_n=2^{n^2}$.
  We will construct a recursive function $g:\mathbb{N}
  \times \mathbb{N} \to \{0,1\}$, so that for all sufficiently large integers $n$ and $j$,
\begin{equation}\label{eq:binary_est_subseq1}
|\frac{1}{t_n}\sum_{k=1}^{t_n}g(k,j) - G(n,j)| < \frac{1}{n}
\end{equation}


  and for all $l \in (t_n, t_{n+1})$,
\begin{equation}\label{eq:binary_est_subseq2}
  \frac{1}{l}\sum_{k=1}^{l}g(k,j) = \theta
  G(n,j) +
  (1-\theta)G(n+1,j) + \epsilon,
\end{equation}
  for some $\theta \in (0,1)$ and $\epsilon \in (-\frac{1}{n},+\frac{1}{n})$.

This lemma will clearly follow from this.

Here is one way to define the function $g$:

For $(n,j) \in \mathbb{N}^2$,  let  $K_{n,j}:=\lfloor
nf(n+1,j)\rfloor $. We thus have $|G(n,j) - K_{n,j}/n|<\frac{1}{n}$.
For $t_n < k \le t_{n+1}$, define $g(k,j)=1$ if
  $(k \mod n) \in \{1,\ldots, K_{n,j}\}$ and $g(k,j)=0$ otherwise.
  Thus,

  $$|\frac{1}{t_{n+1}}\sum_{k=1}^{t_{n+1}}g(k,j) - G(n+1,j)|\le$$

 $$|\frac{1}{t_{n+1}}\sum_{k=1}^{t_{n}}g(k,j)|+\frac{1}{t_{n+1}}\sum_{q=0}^{\frac{t_{n+1}-t_n}{n}}|\sum_{r=0}^{n-1}g(nq+r,j)-G(n,j)|
 + \frac{t_n}{t_{n+1}}|G(n,j)|\le $$

  $$\frac{t_n}{t_{n+1}}+\frac{n}{t_{n+1}}+ \frac{d\cdot t_n}{t_{n+1}} < \frac{1}{n}.$$
The last inequality holding for large $n$'s.

A similar estimate shows that equation \eqref{eq:binary_est_subseq2}
holds.

\end{proof}

\section{\label{sec:entropy_dim_chacterization}An arithmetical characterization of entropy dimensions of SFTs}
\subsection{Statement of main theorem}
We now have the necessary definitions to state our main
result about entropy dimensions of SFTs:
\begin{thm}\label{thm:entrop_dim_char}
  For any $d \ge 2$,
  \begin{enumerate}
    \item{The class of upper entropy dimensions of $\ZD$ SFT's is
$[0,d] \cap \Pi_3$.}
\item{ The class of lower entropy dimensions
of $\ZD$ SFT's is $[0,d] \cap \Sigma_2$.}
\item{ The class of entropy dimensions
of $\ZD$ SFT's is $[0,d] \cap \Delta_2$.}
  \end{enumerate}
\end{thm} 

 Each of the three
statement of theorem \ref{thm:entrop_dim_char} above splits into two
parts. The first part states a necessary restriction on the range of
the entropy dimension: For any SFT the upper entropy dimension is a
number in $[0,d] \cap \Pi_3$. The second part of the statement is a
sufficiency claim about this restriction: For any real number $x \in
[0,d] \cap \Pi_3$ there exists a $\ZZ^{d}$-SFT $X$ with
$\overline{D}(X)=x$.
\subsection{Justification of the arithmetical restrictions}
For $j \ge k$, let $N_{k,j}(X)$ be the number of $k$-blocks
appearing in locally admissible $j$-blocks. It follows that $N_k(X)=
\inf_{j \ge k}N_{k,j}(X)$. The function $f(k,j)=\frac{\log \log
N_{k,j}(X)}{\log k}$ is computable for any SFT $X$. Thus, the upper
and lower/ upper entropy-dimensions of any SFT  $X$ are of the form
$\limsup_{k \to \infty}\inf_j f(k,j)$ and $\liminf_{k \to
\infty}\inf_j f(k,j)$, for a computable function $f:\mathbb{N}^2 \to
\mathbb{Q}$. When the upper and lower entropy dimensions are equal,
they are of the form $\lim_{k \to \infty}\inf_j f(k,j)$. By
\cite{ZW01} numbers of these forms are $\Pi_3$, $\Sigma_2$ and
$\Delta_2$ respectively.

 The restriction
that the entropy dimension of a $\ZD$ subshift is in the range
$[0,d]$, follows from the fact that for any $\ZD$ subshift $X
\subset S^\ZD$ $N_k(X) \le |S|^{k^d}$.

\subsection{ Construction of SFT with given entropy-dimensions - outline }

The proof of the sufficiency part of the claims is constructive in
nature: We describe an algorithm, which given  a ``concrete
description'' of a number in one of the above classes, returns a
$\ZZ^d$-SFT with the appropriate upper/lower entropy dimension.

To simplify the presentation as much as possible, we prove the
result for $d=2$. The generalization of the proof to higher
dimensions introduces no new difficulties.

By lemma \ref{lem:binary_avg_comp}, this amounts, given a computable
function $g:\NN^2 \to \{0,1\}$ to the construction of a $\ZZ^2$ SFT
$X$ with

$$\overline{D}(X) = \limsup_{n \to \infty}\frac{2}{n}\sum_{k=1}^n \inf_{j}g(j,k)=:\alpha_1,$$
and
$$\underline{D}(X) = \liminf_{n \to \infty}\frac{2}{n}\sum_{k=1}^n \inf_{j}g(j,k)=:\alpha_2$$

Here is an overview of this construction: We assume that the
function $g$ is given in terms of a Turing Machine $\mathcal{M}$
which computes its value. The SFT $X_{g}$ to be constructed will be
set up from $2$ ``layers'': a \emph{base layer} and a \emph{control
layer}. The base layer of $X_g$ is essentially of the form
$Y_{\alpha_1,\alpha_2}$ from the example in
\ref{subsec:frac_entrop_dim} above. The control layer is an
extension of a $p$-adic SFT from \ref{subsec:p_adic_SFTs}, where
each admissible point represents arbitrarily long partial \TSD s of
a Turing machine related to $\mathcal{M}$. A maximal
sub-configurations of the control-layer which represents a single
\TSD ~ is a \emph{board}. The local restrictions will be set up so
that in each of the \TSD s represented on a board, the input of the
machine corresponds to segments of the base-layer of the point.
Admissibility of the point will imply that the machine execution
represented in any board does not terminate.

The basic strategy for the construction above was perviously
utilized in \cite{HM_SFT} to construct an SFT with given entropy in
$\Pi_1 \cap [0,\infty)$. Although we bring a self-contained
description of construction here, a reader will probably benefit
from familiarity with the construction of \cite{HM_SFT}.

The term ``boards'' in this context, as well as the idea of
representing partial time-space diagrams of Turing machine
executions goes back to at least to Robinson \cite{robinson71}, in
his proof of Berger's theorem.

In some ways controlling the entropy-dimension requires more care
than just  controlling  the topological entropy: We must take care
that the growth of the control layer, which represents a \TSD~ of a
Turing machine do not effect the entropy-dimension. We overcome this
problem by arranging that boards are very ``sparse'', and so will
have a negligible effect on the growth of $N_k(X_g)$. For this we
use boards which are much sparser then those used in
\cite{robinson71} and \cite{HM_SFT}: In our current construction
boards of level $\ge n$ takes up only $o(n^2)$ cells in a $p^n\times
p^n$ square configuration. Specifically,
 our boards come in levels, such that boards of level $k$ are translates of a set of the form:
 $$\left\{(p^i,p^j) \in \ZZ^2:~  0 < i,j<k \right\} \cup \{(0, p^j):~ 0<j<k\} \cup \{(p^j,0):~ 0<j<k\}.$$

This introduces a further difficulty: Machine executions represented
in the control layer can only access limited and sparse data from
the base layer in order to preform the desired computation. We must
balance things so that this ``limited data'' is sufficient.

Let us turn to the details:

\subsection{\label{ subsec:boards}Sparse representations of Turing machine \TSD s}

Recall that a \emph{Turing machine} $\mathcal{M}$ is given by a
finite set of Machine states $\Sigma_S$ and finite number of tape
characters $\Sigma_T$, along with a transition function $F:\Sigma_S
\times \Sigma_T \to \Sigma_S \times \Sigma_T \times
\{\leftarrow,\rightarrow\}$, indicating the target state, output
character and direction of movent of the machine's head. A
\emph{\TSD} for  $\mathcal{M}$ is an array in $$\left( \Sigma_S\cap
(\Sigma_S \times \Sigma_T \times
\{\leftarrow,\rightarrow\}\right)^{\ZZ^2},$$ where rows represent a
tape configuration and columns represent the evolution of time, in a
manner which is consistent with the transition function $F$. Adding
to each cell a marking of the direction to the head of the machine
in its row (left, or right unless), and including trivial \TSD~
where the machine's head never appears,  the set of all \TSD s for a
given Turing machine $\mathcal{M}$ is an SFT.

We now specify  our implementation of ``boards'':

To begin, we define an SFT $\widehat{X}_p$ which extends $X_p$ for
$p\ge 4$. There will be additional markings on the tiles of
$\widehat{X}_p$. We call the new labels in $\{r,g,y\}^2$, hinting
the traffic light colors red, green and yellow. One color label will
correspond to the horizontal location and the other corresponds to
the vertical location of the node. We will refer to these extra
markings as horizontal and vertical \emph{traffic-light markings}.
 Boards will consists of nodes
with green horizontal and vertical labels. Recall that the position
of each node in $X_p$ corresponds to a pair of $p$-adic integers.
The SFT $\widehat{X}_p$ will have the property that yellow
(horizontal/vertical) labels appear in nodes whose
(horizontal/vertical) position corresponds to a $p$-adic integer
which is equal to $(p-1)p^j+2p^k \mod p^{j+1}$ for some $j \in
\mathbb{N}$ and $k < j$. Green (horizontal/vertical) labels will
appear in nodes whose (horizontal/vertical) position corresponds to
a $p$-adic integer which is equal to $(p-1)p^j \mod p^{j+1}$ for
some $j \in \mathbb{N}$, and will indicate the bottom and left edges
of a board.

The implementation of this is based on a slight generalization of
the construction in subsection \ref{subsec:rat_entropy_dim}.
We
state this in the following lemma:
\begin{lem}\label{lem:matrix_sub_SFT}
  Given $p \ge 2$, a finite set $C$ and a function $A:(p)^2\times C
\to C$, there is an SFT extension $\widehat{X}_p$ of $X_p$ such that
for ever $x \in \widehat{X}_p$ every level-$n$ node with labels
$(i,j) \in (p)^2$ has a label in $c \in C$, so that $c=A(i,j,c')$
where $c' \in C$ is the label of the central node of the
$n+1$-cluster containing this node.
\end{lem}

\begin{proof}
  The tiles of $\widehat{X}_p$ will be those tiles of $X_p$, with an
additional marking in $C$ for every tile. The constraints on the
additional markings are as follows: an arrow with position labels
$(i,j)$ and $C$-marking $c$ can only meet an arrow tail with
$C$-marking $c$ or a perpendicular arrow with $C$-marking $c'$
satisfying $c=A(i,j,c')$.
\end{proof}

In our specific example $C=\{r,g,y\}^2$,
 and the function $A$ is of
the form $A(i,j,(c_i,c_j))=(F(i,c_i),F(j,c_j))$ with $F:(p)\times C
\to C$ defined by:
\begin{enumerate}
  \item{$F(p-1,c)=g ~\forall c \in C$}
  \item{$F(0,g) = g $}
\item{$F(2,g)=y$}
\item{$F(i,c)=r ~\forall c \in C, i\in (p)\setminus\{0,2,p-1\}$}
\item{$F(2,y)=r$}
\end{enumerate}

It follows from these rules that in any $n$-cluster, the
traffic-light markings are determined for those nodes contained in a
cluster whose central node has a (vertical/ horizontal) $p$-marking
different then $0$ or $2$ contained in more then one cluster which
has (vertical/ horizontal) $p$-marking with value $2$. Thus, in a in
a $p^n \times p^n$ configuration of $\widehat{X}_p$ which projects
onto an $n$-cluster there are only $O(n)$ cells which are not
determined. It follows that $N_k(\widehat{X}_p)$ has polynomial
bounded-growth.

Level $1$-nodes which have yellow or green traffic-light markings
are considered free. Cells with green vertical (respectively
horizontal) green traffic-light markings mark the bottom
(respectively left) boundary of a board.

In any $n$-cluster, the traffic-light markings are determined for
those rows and columns containing a digit different then $\{0,1\}$
or more then $1$ digit which is $1$. Thus,in a in a $p^n \times p^n$
configuration of $\widehat{X}_p$ which projects onto an $n$-cluster
there are only $O(n)$ cells which are not determined. It follows
that $N_k(\widehat{X}_p)$ has polynomial bounded-growth.

As in \cite{robinson71} and \cite{HM_SFT}, for $x \in
\widehat{X}_p$, an \emph{infinite board} is an infinite set $A
\subset \ZZ^2$  for which $x\mid_A$ represents an infinite \TSD~
(possibly a trivial one).

The following lemma, summarizes the rest of the ``sparse-boards''
SFT construction:

\begin{lem}\label{lem:SFT_Turing_boards}
Given integers $p,q \ge 2$ and a Turing machine $\mathcal{M}$ with
state set $\Sigma_S$ and tape alphabet $\Sigma_T$ with $S_q \subset
\Sigma_T$, there exists an extension of $Z=Z_{\mathcal{M},p,q}$ of
$\widehat{X}_p \times X_q$ with the following properties:
\begin{enumerate}
  \item{There are factor maps $\pi_p:Z \to \widehat{X}_p$ and $\pi_q: Z \to X_q$.}
  \item{For any point $z \in Z_{\mathcal{M},p,q}$, any board $B$
  in $\pi_p(z)$ represents a \TSD~ of $\mathcal{M}$'s execution, with the bottom row of $B$ representing an a configuration of
  $\mathcal{M}$ in the initial state, with the tape initialized
  according to the corresponding cells of $\pi_q(Z)$.}
  \item{Any any $z \in Z$,  no board represents an execution of $\mathcal{R}$ which reaches the terminating state.}
  \item{ $N_k(Z)$ has polynomial growth.}
\end{enumerate}
\end{lem}
\begin{proof}
Starting from the SFT $\widehat{X}_p \times X_q$ we extend it so
that every board of level $k$ and will represent a $(k+1) \times
(k+1)$ time-space diagram of the Turing machine $\mathcal{M}$:
Coordinates belonging both to a free column
 and a free row will represents cells in this diagram. Free cells will transmit
the \TSD~ information horizontally and vertically whenever there are
no obstruction signals, as in \cite{robinson71}. We superimpose the
restriction that level $1$-nodes with green horizontal traffic-light
markings correspond to  the left boundary of the \TSD~ and tiles
with green vertical traffic-light markings correspond to the lower
boundary of the \TSD~. In particular, a level-$1$ node  with green
horizontal and vertical traffic-light markings  represents a Turing
machine head in the initial state. We also superimpose restrictions
that cells of the bottom row of a board (these are level-$1$ nodes
with green vertical traffic-light markings)  correspond to the tape
initialized according to the corresponding cells of $X_q$. The
extension of $\widehat{X}_p \times X_q$ defined in such a way is
called is $Z$.

It remains to estimate $N_k(Z)$: We already know that
$N_{k}(X_q)=O(k^2)$ and that $N_{k}(\widehat{X}_p)$ is bounded by a
polynomial. Thus, for a $k \times k$ cube, there are $O(k^4)$
possibilities for projecting admissible $Z$-configurations onto
$X_p\times X_q$. Once these projections are fixed, it only remains
to determine the contents of the time-space diagrams represented in
the boards of this configurations. Observe that a time-space diagram
of any Turing machine is completely determined by its bottom row.
Thus, for all boards of level $ \le \log_p k$, the time-space
diagrams are determined by the projection onto $X_q$ of a slightly
larger configuration of size $pk\times pk$. It follows that there
are $O( pk^2)=O(k^2)$ such configurations. The only part of the
configuration which is yet undetermined are those space-time
diagrams corresponding to boards of levels $ > \log_p k$. A most one
such board can intersect a $k \times k$ configuration, and no more
then $(\log_p k)^2$ cells of such board can intersect a $k\times k$
square. Thus, the number of possibilities for this space time
diagram is bounded from above by
 $|S|^{\log_p k}$ where $S$ is the tape and state alphabet of the Turing machine.
This number is polynomial in $k$.

\end{proof}

\subsection{Completing the construction}

By now we have described most of the elements involved in our
construction. We now complete the construction. Mainly, we address
the problem of preforming the computation of $g$ based on ``sparse
inputs'':

\begin{lem}\label{lem:g_density_dim_SFT}
  Let $g:\mathbb{N}^2 \to \{0,1\}$ be a recursive function.
  There
  exist an SFT $X_g$,  such that for any $n \in \NN$,
$$\max_{x \in X_g}|\{ (i,j) :~ 1\le i,j\le 2^k ~ x_{i,j} \mbox{ is a blue node}\}| =\prod_{k=1}^n 4^{a_k},$$
 where
  $a_n = \inf_{j} g(n,j)$. In addition, $X_g$ has the property that $N_k(X_g)$ is bounded by a polynomial in $k$.
\end{lem}
\begin{proof}

Choose some odd integer $p >2$ and let $q=2^p$.

  We build $X_g$ as an extension of an SFT of the form $Z_{\mathcal{M},p,q}$  described in lemma
  \ref{lem:SFT_Turing_boards}. Recall that $Z_{\mathcal{M},p,q}$ is an extension of $X_p \times X_q$. Extend this further,
  by allowing each cross in $X_q$ to contain an extra variable $\alpha \in \{0,1\}$.
  There will be local restrictions which force all $\alpha$ variables corresponding to crosses of the same level to have the same value.
  As in the example of \ref{subsec:frac_entrop_dim}, allow nodes and arrows of the $X_q$ layer to be blue, enforcing the restriction that blue
  arrow head must meet blue arrow tail or a blue perpendicular arrow.
  Also, impose the restriction that a node in $X_q$ can be blue only if both its horizontal and vertical labels are even, or if its
  $\alpha$-variable is equal to $1$. The mechanism of lemma \ref{lem:SFT_Turing_boards} will be used to enforce the condition that the
  $\alpha$-variable is can be set to $1$ only in nodes of levels $n$ for which $a_n=1$. This is implemented as follows:

   Let $\mathcal{R}$ be a Turing machine which performs the following:
   First $\mathcal{R}$ evokes the following procedure, which is intended to compute the level $k$ corresponding to the input:

   Input: $(x_n)_{n \in \mathbb{N}}$
\begin{enumerate}
  \item{If $x_1$ is not a cross, run indefinitely.}

  \item{Otherwise, find the first $m >1 $ such that the horizontal marking of $x_{m+1}$ is not equal to the horizontal marking of $x_{1}$. Set
  $k= \lceil m \frac{\log p}{\log q} \rceil$.}
\end{enumerate}
This procedure indeed finds the desired level $k$, since the
condition on $m$ implies that
$$ p^m < q^k \le p^{m+1},$$
and so
$$ m \log p < k \log q \le (m+1)\log p .$$

Next, $\mathcal{R}$ tries to read the value of the variable $\alpha$
from the input. It is possible if the time-space diagram is on a
board in $X_p$ which is aligned with the correct parity of the $X_2$
layer.
  Then, if $\alpha=0$, $\mathcal{R}$ runs indefinitely. Otherwise,
  the machine $\mathcal{R}$ loops over $j=1,2,\ldots$ and terminate iff
  $g(k,j)=0$ for some $j \in \mathbb{N}$.
  Clearly, for a cross in level $n$, $\alpha$ can be equal to $1$ only if $a_n=1$, and there exists a configuration for which indeed $\alpha=1$
  for every cross of level $n$ with $n=1$.
\end{proof}

To complete the proof of theorem \ref{thm:entrop_dim_char}, extend
the SFT  $X_g$ to an SFT $Y_g$ as follows: Each blue cross an
independent head/tail marking. By following the analysis of the
example in \ref{subsec:frac_entrop_dim}, we conclude that
$$\overline{D}(Y_g) = \limsup_{n \to \infty}\frac{2}{n}\sum_{k=1}^n a_n,$$
and
$$\underline{D}(Y_g) = \liminf_{n \to \infty}\frac{2}{n}\sum_{k=1}^n a_n$$
\section{\label{sec:more} Further remarks on growth-type invariants}

\subsection{Growth-type invariants for topological dynamical systems}
The definition of entropy-dimension given in the beginning of this
paper is a particular case of a definition given by de Carvalho in
\cite{carvalho_97}, which applies to a general action of
$\ZD$-topological dynamical systems:

Let $(X,T)$ be a $\ZD$-dynamical system, e.i $X$ is a compact
topological space, $T_n:X \to X$ is a homeomorphism for each $n \in
\ZD$, so that  $T_n\circ T_m=T_{n+m}$ for $n,m \in \ZD$. For an open
cover $\alpha$, denote by $N(\alpha)$ the cardinality of the
smallest sub-covering of $\alpha$. The (upper) entropy dimension of
the system is defined by:
$$\overline{D}(X,T)=
\inf\left\{ s >0:~ \sup_{\alpha}\limsup_{n \to \infty}\frac{1}{n^s}
\log N(\bigvee_{k \in Q_n}T_n\alpha) =0 \right\}$$

Where $Q_n =\{1,\ldots,n\}^d$, and $\alpha$ ranges over open covers
of $X$.

In case $X \subset S^\ZD$ is a subshift and $\{T_n :~ n \in\ZD\}$
are the shift maps, the open covering of $X$ by cylinders $\alpha_0=
\{ [s]_0 :~ s \in S\}$ satisfies
$$\limsup_{n \to \infty}\frac{1}{n^s}
\log N(\bigvee_{k \in Q_n}T_n\alpha_0) \ge \limsup_{n \to
\infty}\frac{1}{n^s} \log N(\bigvee_{k \in Q_n}T_n\alpha)$$ for
every open cover $\alpha$. Thus the $\sup$ over $\alpha$'s can be
replaced by taking $\alpha_0$.

Now note that for any positive sequence $\{a_n\}_{n \ge 0}$, the
following holds:
  $$\limsup \frac{\log a_n}{\log n} = \inf_{\beta \ge 0}\left\{\beta \ge 0: \limsup \frac{a_n}{n^\beta} > 0\right\}$$

The equivalence of the definitions follows from this.

In a similar manner, it is possible to define polynomial growth-type
and other growth-type invariants for general dynamical systems.

\subsection{Measure-theoretic entropy-dimensions and growth-type invariants}

Measure-theoretic counterparts of growth-type invariants such as
entropy-dimension have been studied by various authors
\cite{blume_97,ferenczi_park_07,katok_thouvenot_97}. In
\cite{carvalho_97} it was shown that for any invariant measure on a
topological system the measure-theoretic entropy-dimension is
bounded from above by the topological entropy-dimension. Answering a
question appearing in \cite{carvalho_97}, Ahn Dou and Park
\cite{ahn_dou_park}  constructed topological systems with positive
topological entropy dimension and zero measure-theoretic entropy
dimension with respect to any invariant measure. We remark that for
the SFTs constructed in the previous section the measure-theoretic
entropy dimension is $0$ for any shift-invariant probability
measure: For any non-trivial $g$,  the density of the blue nodes is
$0$ for any point $x \in X_g$. It follows that with respect to any
invariant probability measure there are (almost) no blue nodes.
Thus, the support of any invariant measure is a subshift with
polynomial growth, hence the entropy dimension of these subshifts is
$0$.

\subsection{Growth-type invariants for effectively closed shifts}

A subshift $X \subset S^\ZD$ is said to be $\Pi_1^0$ if it is
effectively closed - $X$ is the complement of the union of a
recursive sequence of basic open neighborhoods. As observed in
\cite{simpson}, any SFT is a $\Pi_1^0$ subshift. The same argument
we gave in section \ref{sec:entropy_dim_chacterization} shows that
for any $\Pi_1^0$ subshift the upper and lower entropy dimensions
are $\Pi_3$ and $\Sigma_2$ respectively, and the same holds for the
polynomial growth types.

\subsection{Growth-types for Subactions of SFTs}
For any $\ZD$ dynamical system  a \emph{subaction} is an action of
some subgroup of $\ZD$. It is an immediate consequence of the
definition that the entropy dimension of a system is greater or
equal then the entropy dimension of any subaction.

For algebraic $\ZD$-actions, the entropy dimension coincides with
the \emph{entropy rank} - the maximal rank of a subaction with
positive entropy. In particular, algebraic SFTs always have integer
entropy dimension, which is determined by the entropies of
subactions.  For algebraic SFTs, the entropy dimension  is equal to
the Krull dimension of the associated ring. It is also equal to the
unique $k \le d$ such that there is a $k$-dimensional sublattice of
$\ZD$ for which the action has  positive, finite $k$-dimensional
entropy. See \cite{ELMW2001} and related references for details.

For general (non-algebraic) SFTs The presence of a $k$-dimensional
subaction with positive,finite entropy does not imply that the
entropy dimension is $k$: In \cite{TM_CA} it was shown that a
certain endomorphism of the $\ZZ^2$-full-shift admits positive,
finite entropy. The $\ZZ^3$-subshift obtained from all time-space
diagrams of this endomorphism has a $1$-dimensional subaction with
finite positive entropy, yet the entropy rank and the entropy
dimension are $2$.

 In \cite{hochman1} it was shown that any
$\Pi_1^0$-subshift is a isomorphic to a subaction of some
higher-dimensional Sofic shift (a factor of an SFT). This result
gives useful information on the possibilities of  growth-type
invariants for subactions of Sofic shifts.

\subsection{Polynomial growth types of SFTs}
Since

$$\overline{P}(X)= \limsup_{k \to \infty} \inf_j \frac{\log N_{k,j}(X)}{\log k},$$
and
$$\underline{P}(X)= \liminf_{k \to \infty} \inf_j  \frac{\log N_{k,j}(X)}{\log k},$$
with $N_{k,j}(X)$ defined as in section
\ref{sec:entropy_dim_chacterization}, it follows that the upper and
lower polynomial growth types are necessarily $\Pi_3$ and $\Sigma_2$
respectively. Also since $N_k(X)$ is a strictly increasing sequence
of integers for any infinite subshift, any infinite subshift has
polynomial growth order at least $1$. The class of examples in
\ref{subsec:example_polynomial_growth} shows that the possible
values of entropy dimensions are dense in $[2,\infty]$. By a trivial
extension of a $\ZZ$-SFT it is easy to construct a $\ZD$-SFT  with
polynomial growth order $1$.

Using a variation of the construction in section
\ref{sec:entropy_dim_chacterization} combined with the idea of the
example in \ref{subsec:example_polynomial_growth}, we can prove that
any  $x \in [4,\infty]$ which is $\Pi_3$ (or $\Sigma_2$) can be
obtained as an upper (or lower) entropy dimension for some $\ZZ^2$
SFT.

%

 To find SFTs with polynomial growth rate in $(1,2)$ seems to require different methods
then these in the current paper, since the $p$-adic SFTs themselves
have polynomial growth type $2$.


\subsection{Number of Periodic Points for SFTs}
Let $X$ be a $\ZD$ subshift. For any $d$-dimensional lattice $L
\subset \ZD$, let $P_L(X)$ denote the number of fixed points of $X$
under $L$. Unlike $N_k(X)$, this is indeed an invariant of $X$.
Furthermore, it is a computable invariant: There exists an algorithm
which given a set of tiles and a lattice $L$ calculates $P_L(X)$ for
the SFT associated with these tiles. A naive algorithm of this type
will have running time which is exponential in $|\det L|$.

For $k \in \mathbb{N}$ set $P_k(X)=P_{k\ZD}(X)$.

In dimension $1$, $P_k(X)=\mathit{trace}(A^k)$ for a matrix $A$
associated with $X$. In dimension $d >1$, it is not likely to have a
simple formula for $P_k(X)$.

It should be possible to prove a lower-bound for the computation
time of $P_k(X)$ which is exponential in $\log k$:

The problem of determining if a given Turing machine runs into a
loop of period $k$ has a lower bound which is exponential in $\log
k$. Using SFT representation for \TSD~ of Turing machines, one can
reduce the loop-problem for a Turing machine into counting
$L$-periodic points of certain SFTs.

A remarkable paper \cite{kim_ornes_roush2000} of Kim, Ormes and
Roush gives a simple necessary and sufficient conditions on an
$n$-tuple of complex numbers to be the non-zero spectra of a matrix
representing an SFT $X$. This gives a relatively simple
characterization of the possible sequences $\{P_k(X)\}$ for
one-dimensional SFTs.

It is interesting to understand the possible asymptotics for
$\{P_k(X)\}$ when $X$ is an arbitrary multidimensional SFT.

\bibliographystyle{abbrv}
\bibliography{entropy_dim}
\end{document}